\documentclass[12pt,a4paper]{amsart}
\usepackage[english]{babel}
\usepackage[T1]{fontenc}
\usepackage{amsmath,amssymb}
\usepackage{calc}
\usepackage{float}
\usepackage{xcolor}
\usepackage{tikz}
\usepackage[colorlinks=true,linkcolor=blue,anchorcolor=blue,citecolor=blue,pagebackref,linktocpage]{hyperref}

\theoremstyle{plain}
\newtheorem{theorem}{Theorem}[section]
\newtheorem{proposition}[theorem]{Proposition}
\newtheorem{lemma}[theorem]{Lemma}
\newtheorem{corollary}[theorem]{Corollary}
\newtheorem{problem}[theorem]{Problem}
\theoremstyle{remark}
\newtheorem{remark}[theorem]{Remark}

\newcommand{\arxiv}[1]{\href{https://arxiv.org/abs/#1}{{\tt arXiv:#1}}}
\newcommand{\Aut}{\operatorname{Aut}}
\newcommand{\NE}{\overline{\operatorname{NE}}}
\newcommand{\Nef}{\operatorname{Nef}}
\newcommand{\NS}{\operatorname{NS}}

\begin{document}

\title[Mori dream surfaces of Kodaira dimension one]{Mori dream Jacobian elliptic surfaces of Kodaira dimension one}
\author{Antonio Laface}
\address{Departamento de Matem\'atica, Universidad de Concepci\'on, Casilla 160-C, Concepci\'on, Chile}
\email{\href{mailto:alaface@udec.cl}{alaface@udec.cl}}

\author{Sichen Li}
\address{
School of Mathematics, East China University of Science and Technology, Shanghai 200237, P. R. China}
\email{\href{mailto:sichenli@ecust.edu.cn}{sichenli@ecust.edu.cn}}
\author{Jihao Liu}
\address{
Department of Mathematics, Peking University, No. 5 Yiheyuan Road, Haidian District, Beijing 100871, China
\endgraf
Beijing International Center for Mathematical Research, Peking University, No. 5 Yiheyuan Road, Haidian District, Beijing 100871, China}
\email{\href{mailto:liujihao@math.pku.edu.cn}{liujihao@math.pku.edu.cn}}
\begin{abstract}
Let $\pi\colon X\to\mathbb P^1$ be a Jacobian elliptic surface over
$\mathbb C$, and set $\chi=\chi(\mathcal O_X)\ge3$, so that
$\kappa(X)=1$. Assume that the Mordell--Weil group of $\pi$ is finite and that $\pi$
has at least one reducible fiber, the reducible fibers being of types
$I_{n_1},\ldots,I_{n_s}$. We prove that the zero section and the components
of the reducible fibers generate the closed Mori cone if and only if
\[
 \sum_{i=1}^s\frac{\lfloor n_i^2/4\rfloor}{n_i}\le\chi.
\]
If $\sum_i n_i\le2\chi+3$, then $X$ is a Mori dream surface. The proof
combines an explicit description of the facets of the cone generated by the
curves visible in the fibration with Artin's criterion applied to the null
loci of the dual nef rays. We also show that, in Kodaira dimension one,
finiteness of both the Mordell--Weil group and the automorphism group does not
imply polyhedrality of the Mori cone. In the polyhedral range, we construct a
Jacobian elliptic surface with $(\chi,n)=(3,11)$, Picard number $12$, and a
big and nef divisor which is not semiample; in particular, this surface is not
a Mori dream surface. Finally, for every integer $\rho\ge2$, we construct a
Jacobian elliptic surface of Kodaira dimension one and Picard number $\rho$
that is a Mori dream surface.
\end{abstract}

\maketitle

\section{Introduction}

Mori dream spaces were introduced by Hu and Keel \cite{HK00} as varieties on
which the main steps of the minimal model program can be organized into
finitely many birational models. Equivalently, under the standard hypotheses,
their Cox rings are finitely generated. For a smooth projective surface $X$
with $q(X)=0$, the criterion takes a particularly concrete form: $X$ is a
Mori dream surface if and only if its effective cone is rational polyhedral
and every nef divisor is semiample; see \cite[Corollary~2.6]{AHL10} and
\cite[Section~5.1]{ADHL15}.

The classification of Mori dream surfaces is far from complete. Several
important classes are nevertheless understood. Smooth rational surfaces with
big anticanonical class are Mori dream surfaces \cite{TVV11}; for smooth
rational surfaces with $\kappa(X,-K_X)=1$, finite generation of the Cox ring is
equivalent to rational polyhedrality of the effective cone \cite{AL11}.
Recent work has also refined the description of Cox-ring generators for
rational surfaces with nef anticanonical class \cite{AG25}. Projective K3
surfaces are Mori dream precisely when their automorphism groups are finite
\cite{AHL10,Sterk85}. In Kodaira dimension two, Keum and Lee constructed
minimal Mori dream surfaces of general type with $p_g=0$ and
$2\le K_X^2\le9$ \cite{KL19}.

Elliptic surfaces of Kodaira dimension one occupy the gap between these
settings. For a rational Jacobian elliptic surface, finite generation of the
Cox ring is equivalent to finiteness of the Mordell--Weil group, and Cox rings
of the extremal cases have been described explicitly
\cite{AL11,AGL16}. No analogous general criterion appears to be available for
properly elliptic surfaces. To the best of our knowledge, no Mori dream
surface of Kodaira dimension one had previously been exhibited. The existence
problem was raised in \cite[Section~7.2]{KL19} and explicitly posed for
elliptic surfaces over $\mathbb P^1$ in \cite[Question~5.15]{HL26}.

We work with a Jacobian elliptic surface
\[
 \pi\colon X\longrightarrow\mathbb P^1
\]
over $\mathbb C$ and write $\chi=\chi(\mathcal O_X)$. When $\chi\ge3$, one
has $q(X)=0$, $K_X\sim(\chi-2)F$, and $\kappa(X)=1$, where $F$ is a general
fiber. Every section has self-intersection $-\chi$. Consequently, rational
polyhedrality of $\NE(X)$ forces the Mordell--Weil group to be finite. One of
the points of this paper is that the converse fails, even if the automorphism
group is finite; see Theorem~\ref{Not-polyhedral-thm}. Thus two distinct issues
must be addressed: finite generation of the Mori cone and semiampleness of its
extremal nef classes.

Our first result concerns fibrations with possibly several reducible fibers.
Assume that these fibers are $F_1,\ldots,F_s$, of types
$I_{n_1},\ldots,I_{n_s}$, and let $\mathcal T$ be the set of their irreducible
components. If $C_0$ denotes the zero section, we consider the
\emph{fibration cone}
\[
 \mathcal C_\pi:=\mathbb R_{\ge0}[C_0]
 +\sum_{T\in\mathcal T}\mathbb R_{\ge0}[T].
\]
We also set
\[
 \delta(\pi):=\sum_{i=1}^s
 \frac{\lfloor n_i^2/4\rfloor}{n_i},
 \qquad
 N(\pi):=\sum_{i=1}^s n_i.
\]
The proof uses only that every reducible fiber is of type $I_n$; in
particular, the result applies to every semistable Jacobian elliptic surface.

\begin{theorem}
\label{Jac-thm}
Let $\pi\colon X\to\mathbb P^1$ be a Jacobian elliptic surface with
$\chi:=\chi(\mathcal O_X)\ge3$ and finite Mordell--Weil group. Suppose that
$\pi$ has at least one reducible fiber and that every reducible fiber is of
type $I_n$. Then the following statements hold.
\begin{enumerate}
\item The equality
\[
 \mathcal C_\pi=\NE(X)
\]
holds if and only if $\delta(\pi)\le\chi$. In particular, the Mori cone is
rational polyhedral in this range.
\item If $N(\pi)\le2\chi+3$, then $X$ is a Mori dream surface.
\end{enumerate}
\end{theorem}

The two assertions reflect the two parts of the Mori dream criterion. The
Shioda--Tate formula describes the N\'eron--Severi space in terms of the zero
section, a fiber, and the root lattices of the reducible fibers. A
nonvertical facet of $\mathcal C_\pi$ is obtained by choosing one component in
each reducible fiber and removing the chosen components from the list of
generators. Its supporting class is obtained from $C_0+\chi F$ by one local
correction in each reducible fiber. If $S$ denotes the chosen collection of
components, the square of the supporting class is $\chi-\delta(S)$, where
$\delta(S)$ is the sum of the corresponding local defects. This leads to the
exact condition $\delta(\pi)\le\chi$ in the first assertion.

For the second assertion, the same description identifies the null loci of
the nonvertical nef rays. In the range $N(\pi)\le2\chi+3$, every positive
cycle supported on one of these null loci has arithmetic genus at most zero.
Artin's rationality criterion and the block theorem then imply semiampleness
of all extremal nef classes.

A useful consequence can be stated solely in terms of the Picard number.

\begin{corollary}
\label{small-rho-cor}
Under the assumptions of Theorem~\ref{Jac-thm}, if
\[
 \rho(X)\le\chi(\mathcal O_X)+3,
\]
then $X$ is a Mori dream surface. In particular, every such surface with
$\rho(X)\le6$ is a Mori dream surface.
\end{corollary}

Thus Theorem~\ref{Jac-thm}(2) also applies to many semistable fibrations
with several reducible fibers, although the examples constructed later use
one-fiber configurations.

The first assertion of Theorem~\ref{Jac-thm} also makes clear where the
remaining semiampleness problem begins. If $\delta(\pi)<\chi$, every
nonvertical extremal nef ray is big. At the boundary $\delta(\pi)=\chi$,
however, isotropic extremal nef rays may occur, and their Iitaka dimension is
not understood in general.

\begin{problem}
\label{isotropic-nef-problem}
Assume the hypotheses of Theorem~\ref{Jac-thm} and suppose that
\[
 \delta(\pi)=\chi.
\]
Let $D$ span a nonvertical extremal ray of $\Nef(X)$ with $D^2=0$.
Is $\kappa(X,D)>0$? In particular, is $D$ semiample?
\end{problem}

When there is a unique reducible fiber, Theorem~\ref{Jac-thm} specializes to
the following.

\begin{corollary}
\label{one-fiber-cor}
Suppose, in addition to the hypotheses of Theorem~\ref{Jac-thm}, that
$\pi$ has exactly one reducible fiber, of type $I_n$. Then the following hold:
\begin{enumerate}
\item
$\mathcal C_\pi=\NE(X)$ if and only if $n\le4\chi$;
\item
if $n\le2\chi+3$, then $X$ is a Mori dream surface.
\end{enumerate}
\end{corollary}

The second bound is sharp for the uniform Artin argument: once
$n\ge2\chi+4$, some nonvertical facet contains a positive cycle of arithmetic
genus one. Thus the range between the Artin bound and the cone bound is the
first place where a new semiampleness obstruction can occur. For $\chi=3$,
the first case not settled by the uniform criterion is $n=10$. The next case
already shows that semiampleness can fail.

\begin{theorem}
\label{I11-thm}
There exists a Jacobian elliptic surface
$\pi\colon X\to\mathbb P^1$ over $\mathbb C$ with
$\chi(\mathcal O_X)=3$, $\kappa(X)=1$, $\rho(X)=12$, and singular-fiber
configuration $I_{11}+25I_1$. If $C_0$ is its zero section and
$F=C_1+\cdots+C_{11}$ is the reducible fiber, then
\[
 \NE(X)=\sum_{i=0}^{11}\mathbb R_{\ge0}[C_i].
\]
Moreover, $X$ admits a big and nef divisor which is not semiample. In
particular, $X$ is not a Mori dream surface.
\end{theorem}

The proof identifies a minimally elliptic cycle in the null locus of a dual
nef ray and computes a nonzero additive obstruction to triviality on that
cycle for an explicit Weierstrass model. We then show that the obstruction
varies regularly over the cover on which the components of the
$I_{11}$-fiber are labeled, and that the $I_{11}+25I_1$ stratum is pure of
dimension $18$. The Noether--Lefschetz estimate therefore provides a member
of Picard number $12$ on which the obstruction remains nonzero.

Our main application answers the existence question in every Picard number.

\begin{theorem}
\label{Mainthm}
For every integer $\rho\ge2$, there exists a Jacobian elliptic surface $X$
with $\kappa(X)=1$ and $\rho(X)=\rho$ that is a Mori dream surface.
\end{theorem}

For $\rho=2$, we use very general Jacobian elliptic surfaces whose
N\'eron--Severi group is generated by the zero section and a fiber. For
$\rho\ge3$, we start from a very general elliptic K3 surface and perform a
suitably chosen base change of degree $\rho-1$, totally ramified over one
nodal fiber. The resulting surface has a unique reducible fiber of type
$I_{\rho-1}$ and trivial Mordell--Weil group, so
Theorem~\ref{Jac-thm} applies.

Since a rational polyhedral Mori cone implies bounded negativity
\cite[Proposition~1.1]{AL11}, the examples above also provide elliptic
surfaces of Kodaira dimension one satisfying the bounded negativity
conjecture \cite[Conjecture~1.1]{Bauer et al 2013}. They motivate the broader
problem:
\begin{problem}
Classify Mori dream surfaces of Kodaira dimension one.
\end{problem}

The paper is organized as follows. Section~\ref{Pre} records the surface and
elliptic-fibration tools used later. We compare polyhedrality with finiteness
of the Mordell--Weil and automorphism groups in
Section~\ref{polyhedral-section}. Section~\ref{proof-section} describes the
fibration cone and proves Theorem~\ref{Jac-thm}, including the sharp
one-fiber form of the Artin bound. Section~\ref{I11-section} proves
Theorem~\ref{I11-thm}, and Section~\ref{Example-Sect} constructs the examples
required for Theorem~\ref{Mainthm}.

\subsection*{Acknowledgments}
Antonio Laface was partially supported by Proyecto FONDECYT Regular n.~1230287.
Sichen Li would like to thank Guolei Zhong for many helpful conversations and
Jonghae Keum for kindly answering his questions. Jihao Liu was partially
supported by the National Key R\&D Program of China
(grant no.~2024YFA1014400).

\section{Preliminaries}
\label{Pre}
We work over $\mathbb C$. By a curve we mean a reduced and irreducible curve.
For a projective variety $X$, we write $\NE(X)$ for the closed Mori cone,
$\Nef(X)$ for the nef cone, and $\mathrm{SAmp}(X)$ for the semiample cone. If
$D$ is a Cartier divisor, then $\mathbf B(D)$ denotes its stable base locus.

\subsection{Mori dream surfaces}
For the surfaces considered here, the Mori dream property is characterized by
polyhedrality and semiampleness.

\begin{proposition}
\label{char-MDS-prop}
Let $X$ be a smooth projective surface with $q(X)=0$. Then $X$ is a Mori dream
surface if and only if $\NE(X)$ is rational polyhedral and
$\Nef(X)=\mathrm{SAmp}(X)$.
\end{proposition}

\begin{proof}
This follows from \cite[Corollary~2.6]{AHL10} and
\cite[Remark~2.2.27]{Lazarsfeld04}.
\end{proof}

We shall repeatedly reduce semiampleness to the extremal rays of a polyhedral
nef cone.

\begin{proposition}
\label{NE-reduced-prop}
Let $X$ be a projective variety whose nef cone is rational polyhedral, and let
$H_1,\ldots,H_r$ be integral generators of its extremal rays. Then
\[
 \Nef(X)=\mathrm{SAmp}(X)
\]
if and only if every $H_i$ is semiample.
\end{proposition}

\begin{proof}
One implication is immediate. Conversely, after multiplication by a positive
integer, every integral nef divisor is a nonnegative integral sum of the
$H_i$. A common positive multiple of this sum is globally generated.
\end{proof}

\subsection{Jacobian elliptic surfaces}
An elliptic fibration is a surjective morphism with connected fibers from a
smooth projective surface to a smooth projective curve whose general fiber is
an elliptic curve. All elliptic fibrations in this paper are relatively
minimal. A Jacobian elliptic fibration is an elliptic fibration with a chosen
zero section; its sections form the Mordell--Weil group.

\begin{proposition}
\label{kod1-basic-prop}
Let $\pi\colon X\to\mathbb P^1$ be a Jacobian elliptic surface with
$\chi:=\chi(\mathcal O_X)>0$, let $F$ be a general fiber, and let $C$ be a
section. Then
\[
 q(X)=0,\qquad p_g(X)=\chi-1,\qquad C^2=-\chi,
 \qquad K_X\sim(\chi-2)F.
\]
In particular, $\kappa(X)=1$ if and only if $\chi\ge3$.
\end{proposition}

\begin{proof}
The equality $q(X)=0$ follows from
\cite[Chapter~7, Lemma~14, p.~176]{Friedman98}, and hence
$p_g(X)=\chi-1$. The canonical bundle formula gives
$K_X\sim(\chi-2)F$; see \cite[Theorem~5.44]{SS19}. Finally,
$C^2=-\chi$ by \cite[Corollary~5.45]{SS19}. The assertion about the Kodaira
dimension follows immediately.
\end{proof}

We also use the following consequence of the Shioda--Tate formula.

\begin{proposition}
\label{jac-rho-prop}
Let $\pi\colon X\to B$ be a Jacobian elliptic fibration. Let
$F_1,\ldots,F_s$ be its reducible fibers, and write
$F_i=C_{i,0}+\cdots+C_{i,m_i-1}$, where $C_{i,0}$ is met by the zero section
$C_0$. If $\mathrm{MW}(\pi)$ is finite, then
\[
 \rho(X)=2+\sum_{i=1}^s(m_i-1),
\]
and the classes of $C_0$, a general fiber $F$, and the components $C_{i,j}$
with $1\le i\le s$ and $1\le j\le m_i-1$ form a basis of
$\NS(X)_{\mathbb Q}$.
\end{proposition}

\begin{proof}
The rank formula follows from Shioda--Tate
\cite[Corollary~6.7]{SS19}. Suppose that
\[
 D=aC_0+bF+\sum_{i=1}^s\sum_{j=1}^{m_i-1}c_{i,j}C_{i,j}\equiv0.
\]
Intersecting with $F$ gives $a=0$. For each $i$, the intersection matrix of
$C_{i,1},\ldots,C_{i,m_i-1}$ is negative definite, so $D^2=0$ forces every
$c_{i,j}$ to vanish. Intersecting with $C_0$ then gives $b=0$.
\end{proof}

\begin{remark}
Proposition~\ref{jac-rho-prop} gives a basis of the N\'eron--Severi space; it
does not assert that the zero section and the fiber components generate the
Mori cone. The exact condition for this equality is established in
Theorem~\ref{Jac-thm}(1).
\end{remark}

\subsection{Blocks and semiampleness}
Let $D$ be a nef and big divisor on a smooth projective surface. Its null locus
is finite, and the intersection form on the classes of its components is
negative definite by the Hodge index theorem. A \emph{block} for $D$ is an
effective, possibly nonreduced divisor $B$ supported precisely on this null
locus and satisfying $B\cdot C<0$ for each component $C$. Such blocks exist;
see \cite[Section~5.1.2]{ADHL15}.

\begin{proposition}
\label{block-semi-prop}
Let $D$ be a nef and big divisor on a smooth projective surface, and let $B$ be
a block for $D$. Then
\[
 \mathbf B(D)=\mathbf B(D|_B),
\]
where the right-hand side is the stable base locus of the restrictions
$\mathcal O_B(mD)$. In particular, if $R$ is a connected component of the
possibly nonreduced divisor $B$ and $H^1(R,\mathcal O_R)=0$, then
\[
 R\cap\mathbf B(D)=\varnothing.
\]
\end{proposition}

\begin{proof}
This is \cite[Proposition~5.1.2.5 and Remark~5.1.2.6]{ADHL15}.
\end{proof}

\section{Polyhedrality of the Mori cone and finiteness of the Mordell--Weil
and automorphism groups}
\label{polyhedral-section}
We first compare polyhedrality of the Mori cone with finiteness of the
Mordell--Weil and automorphism groups. In Kodaira dimension one,
polyhedrality forces both groups to be finite, but the converse fails.
\begin{proposition}
Let $\pi\colon X\to\mathbb P^1$ be a Jacobian elliptic surface with
$\kappa(X)=1$ and $\chi(\mathcal O_X)>0$. If $\NE(X)$ is rational
polyhedral, then both $\Aut(X)$ and $\mathrm{MW}(\pi)$ are finite.
\end{proposition}
\begin{proof}
We first note that $\Aut_0(X)=\{\mathrm{id}\}$ as in \cite[Section 2.1]{CLS24}.
The action of $\Aut(X)$ permutes the finitely many extremal rays of $\NE(X)$.
The subgroup fixing every ray has finite index and fixes an ample class.
A Fujiki--Lieberman-type theorem (cf. \cite[Theorem 1.4]{Li20}) therefore
implies that this subgroup is finite, and hence $\Aut(X)$ is finite.

Every section $C$ satisfies $C^2=-\chi(\mathcal O_X)<0$ by
Proposition~\ref{kod1-basic-prop}. If $\mathrm{MW}(\pi)$ were infinite, its
elements would give infinitely many negative curves spanning distinct
extremal rays of $\NE(X)$, contradicting polyhedrality. Therefore,
$\mathrm{MW}(\pi)$ is finite.
\end{proof}
We now show that the converse fails.
\begin{theorem}
\label{Not-polyhedral-thm}
There exists a Jacobian elliptic surface
$\pi\colon X\to\mathbb P^1$ with $\kappa(X)=1$ such that $\NE(X)$ is not
rational polyhedral, whereas both $\mathrm{MW}(\pi)$ and $\Aut(X)$ are finite.
\end{theorem}
\begin{proof}
Let \(\pi_Y\colon Y\to B\simeq\mathbb P^1\) be a semistable extremal
elliptic K3 surface with singular fiber configuration
\[
I_{16}+I_4+4I_1
\]
and Mordell--Weil group
\(\operatorname{MW}(\pi_Y)\simeq\mathbb Z/4\mathbb Z\). The existence of
such a fibration is established by Shimada
\cite[Theorem 2.6(1)]{Shimada00}. Its trivial lattice has rank
\[
2+(16-1)+(4-1)=20,
\]
so \(Y\) is a singular K3 surface. By Shioda--Inose
\cite[Proof of Theorem 5, p. 130]{SI77}, \(Y\) admits another elliptic pencil
with infinite Mordell--Weil group. Its infinitely many sections are rational
\((-2)\)-curves. Since \(\pi_Y\) has only finitely many vertical components
and finitely many sections, infinitely many of these curves are multisections
of \(\pi_Y\).

Let \(\beta\colon B'\to B\) be the double cover ramified over the points
supporting the \(I_{16}\)- and \(I_4\)-fibers. Then
\(B'\simeq\mathbb P^1\). Let \(W\) be the normalization of
\(Y\times_B B'\), and let
\[
p\colon W\to Y,\qquad q\colon X\to W
\]
be the natural finite morphism and the minimal resolution, respectively. The
induced Jacobian elliptic fibration
\(\pi_X\colon X\to B'\) has singular fiber configuration
\[
I_{32}+I_8+8I_1.
\]
Consequently, \(c_2(X)=48\). Since \(K_X^2=0\), Noether's formula gives
\(\chi(\mathcal O_X)=4\), and the canonical bundle formula gives
\[
K_X\sim2F.
\]
Thus \(\kappa(X)=1\). Moreover, \(q(X)=0\), \(p_g(X)=3\), and hence
\(h^{1,1}(X)=40\). The trivial lattice of \(\pi_X\) has rank
\[
2+(32-1)+(8-1)=40.
\]
It follows that \(\rho(X)=40\), and the Shioda--Tate formula yields
\(\operatorname{rank}\operatorname{MW}(\pi_X)=0\). Therefore
\(\operatorname{MW}(\pi_X)\) is finite.

We next prove that \(\NE(X)\) is not rational polyhedral. Let
\(C\subset Y\) be one of the infinitely many \((-2)\)-multisections above,
and set \(A=p^*C\). Since \(p\) is finite of degree two, the projection
formula gives
\[
A^2=2C^2=-4.
\]
Write
\[
q^*A=D+E,
\]
where \(D\) is the total strict transform of \(A\) and \(E\) is an effective
\(q\)-exceptional divisor. The projection formula gives
\(q^*A\cdot E=0\), and hence
\[
D^2=(q^*A)^2+E^2=-4+E^2\le-4,
\]
because the intersection form on the exceptional locus is negative definite.
Every irreducible component of \(D\) is horizontal and maps onto \(C\). If
all these components had nonnegative self-intersection, then \(D^2\ge0\),
since distinct irreducible curves have nonnegative intersection. Thus \(D\)
contains a horizontal negative curve. Curves obtained from distinct
multisections \(C\) are distinct because their images on \(Y\) are distinct.
Every negative irreducible curve spans an extremal ray of the Mori cone, and
distinct negative curves span distinct rays. Therefore \(X\) has infinitely
many extremal rays, so \(\NE(X)\) is not rational polyhedral.

It remains to prove that \(\Aut(X)\) is finite. Since \(K_X\sim2F\), the
fibration \(\pi_X\) is the Iitaka fibration of \(X\). Hence every automorphism
of \(X\) preserves \(\pi_X\), and there is an exact sequence
\[
1\longrightarrow\Aut_{B'}(X)\longrightarrow\Aut(X)
\longrightarrow G\longrightarrow1,
\]
where \(G\subseteq\Aut(B')\) is the subgroup induced on the base. The group
\(G\) preserves the finite set of singular values of \(\pi_X\), which
contains at least three points. Its stabilizer in
\(\operatorname{PGL}_2(\mathbb C)\) is finite, so \(G\) is finite.

Let \(K=\mathbb C(B')\), and let \(E/K\) be the generic fiber of \(\pi_X\).
Restriction to the generic fiber gives an injection
\[
\Aut_{B'}(X)\hookrightarrow\Aut_K(E).
\]
Since \(E\) has an origin,
\[
\Aut_K(E)=E(K)\rtimes\Aut_K(E,0).
\]
Here \(E(K)=\operatorname{MW}(\pi_X)\) is finite, and
\(\Aut_K(E,0)\) is finite. Thus \(\Aut_{B'}(X)\) is finite. Since both the
kernel and the image are finite, \(\Aut(X)\) is finite.
\end{proof}

\section{The fibration cone and proof of Theorem~\ref{Jac-thm}}
\label{proof-section}
\subsection{Facets of the fibration cone}
In the rest of this subsection, let
$\pi\colon X\to\mathbb P^1$ be a Jacobian elliptic surface with
$\chi=\chi(\mathcal O_X)\ge3$ and finite Mordell--Weil group. Assume that the
reducible fibers are $F_1,\ldots,F_s$, of types
$I_{n_1},\ldots,I_{n_s}$, where $s\ge1$ and $n_i\ge2$. Let $C_0$ be the zero
section. For every $i$, write
\[
 F_i=C_{i,0}+\cdots+C_{i,n_i-1}
\]
with cyclically ordered components, numbered so that $C_0$ meets $C_{i,0}$.
Let $\mathcal T$ be the set of all these components, and retain the
fibration cone $\mathcal C_\pi$ introduced before Theorem~\ref{Jac-thm}.

A \emph{selection} is a subset $S\subset\mathcal T$ containing exactly one
component of each reducible fiber. We associate with $S$ the cone
\[
 \mathcal F_S:=\mathbb R_{\ge0}[C_0]
 +\sum_{T\in\mathcal T\setminus S}\mathbb R_{\ge0}[T],
\]
and we set
\[
 \mathcal V_\pi:=\sum_{T\in\mathcal T}\mathbb R_{\ge0}[T].
\]

\begin{proposition}
\label{fibration-facets-prop}
The facets of $\mathcal C_\pi$ are precisely the vertical facet
$\mathcal V_\pi$ and the cones $\mathcal F_S$, where $S$ runs through all
selections. The vertical facet need not be simplicial, whereas every
$\mathcal F_S$ is simplicial.
\end{proposition}

\begin{proof}
Let $D\in\mathcal C_\pi^\vee$. Thus
\[
 D\cdot C_0\ge0,\qquad D\cdot T\ge0\quad(T\in\mathcal T).
\]
These intersection numbers determine $D$. Indeed, if they all vanish, then
$D\cdot F=0$, and Proposition~\ref{jac-rho-prop} shows successively that the
coefficients of $C_0$, of the root lattices, and of $F$ all vanish.

Put $x=D\cdot C_0$. For every reducible fiber $F_i$, one has
\[
 \sum_{T\subset F_i}D\cdot T=D\cdot F.
\]
If $D\cdot F=0$, all these numbers vanish. Then $D-xF$ has zero intersection
with $C_0$ and with every component in $\mathcal T$, so $D=xF$. This gives
the extremal ray $\mathbb R_{\ge0}F$, dual to $\mathcal V_\pi$.

Suppose that $D\cdot F>0$ and normalize this intersection to one. If $x>0$,
then
\[
 D=xF+(D-xF)
\]
is a nontrivial decomposition in $\mathcal C_\pi^\vee$. Hence a nonvertical
extremal ray must satisfy $x=0$. On each reducible fiber, the numbers
$D\cdot T$ are nonnegative and have sum one. They therefore form a point of a
simplex. Taking all reducible fibers together gives a product of simplices,
whose vertices are obtained by choosing one component in each fiber. Duality
between facets and extremal rays gives the asserted list.

Finally, removing one component from an $I_{n_i}$-fiber leaves a negative
definite chain. Intersecting a possible relation among the generators of
$\mathcal F_S$ with $F$ first removes the coefficient of $C_0$; negative
definiteness on each remaining chain then removes all other coefficients.
Thus $\mathcal F_S$ is simplicial.
\end{proof}

\subsection{The supporting classes}
Let $G$ be a reducible fiber of type $I_n$, and let $E_G$ be the component
met by $C_0$. For a component $T\subset G$, let $d(T)$ be its cyclic distance
from $E_G$, chosen between $0$ and $\lfloor n/2\rfloor$, and define
\[
 \delta(T):=\frac{d(T)(n-d(T))}{n}.
\]
In particular, $\delta(E_G)=0$. If $T\ne E_G$, let $\omega_T$ be the class in
the rational root lattice generated by the components of $G$ not meeting
$C_0$, characterized by
\[
 \omega_T\cdot T=-1,
 \qquad
 \omega_T\cdot T'=0
\]
for every other nonidentity component $T'\subset G$. We put
$\omega_{E_G}=0$.

\begin{lemma}
\label{local-correction-lem}
For every component $T\subset G$, one has
\[
 \omega_T^2=-\delta(T),
 \qquad
 \omega_T\cdot E_G=
 \begin{cases}
 1,&T\ne E_G,\\
 0,&T=E_G.
 \end{cases}
\]
If $\Gamma$ is an irreducible horizontal curve and $m=\Gamma\cdot F$, then
\[
 \omega_T\cdot\Gamma\le m\delta(T).
\]
\end{lemma}

\begin{proof}
The assertion is immediate for $T=E_G$. Otherwise write
$G=E_0+\cdots+E_{n-1}$ cyclically, with $E_0=E_G$ and $T=E_k$. If $A$ is
the Cartan matrix of type $A_{n-1}$, then
\[
 \omega_T=\sum_{j=1}^{n-1}(A^{-1})_{jk}E_j,
 \qquad
 (A^{-1})_{jk}
 =\frac{\min\{j,k\}(n-\max\{j,k\})}{n}.
\]
The diagonal entry is $k(n-k)/n=\delta(T)$, which gives
$\omega_T^2=-\delta(T)$. Since $\omega_T\cdot G=0$ and its intersections
with the nonidentity components sum to $-1$, one obtains
$\omega_T\cdot E_G=1$.

For the last assertion, put $q_j=\Gamma\cdot E_j$. The entries in the
$k$-th column of $A^{-1}$ are nonnegative and bounded above by
$\delta(T)$. Since $\sum_{j=1}^{n-1}q_j\le m$, it follows that
\[
 \omega_T\cdot\Gamma
 =\sum_{j=1}^{n-1}(A^{-1})_{jk}q_j
 \le m\delta(T).
\]
\end{proof}

For a selection $S$, set
\[
 \delta(S):=\sum_{T\in S}\delta(T).
\]
For a fiber of type $I_n$, the local defect
\[
 \frac{d(n-d)}{n}
\]
is maximal for $d=\lfloor n/2\rfloor$, and its maximum is
$\lfloor n^2/4\rfloor/n$. Since the choice of a component can be made
independently in each reducible fiber, it follows that
\[
 \delta(S)\le\delta(\pi)
\]
for every selection $S$, and equality holds for a selection which maximizes
the local defect in every fiber. Equivalently,
\[
 \delta(\pi)=\max_S\delta(S).
\]

Define
\[
 H_S:=C_0+\chi F-\sum_{T\in S}\omega_T.
\]

\begin{proposition}
\label{supporting-class-prop}
The class $H_S$ satisfies
\[
 H_S\cdot C_0=0,
 \qquad
 H_S\cdot T=
 \begin{cases}
 1,&T\in S,\\
 0,&T\in\mathcal T\setminus S,
 \end{cases}
 \qquad
 H_S\cdot F=1.
\]
Thus $H_S$ supports the facet $\mathcal F_S$. Moreover,
\[
 H_S^2=\chi-\delta(S).
\]
\end{proposition}

\begin{proof}
The intersection statements follow from
Lemma~\ref{local-correction-lem}. Corrections belonging to distinct fibers
are orthogonal, and every $\omega_T$ is orthogonal to both $C_0$ and $F$.
Since $(C_0+\chi F)^2=\chi$, the formula for $H_S^2$ follows.
\end{proof}

\begin{proof}[Proof of Theorem~\ref{Jac-thm}(1)]
The inclusion $\mathcal C_\pi\subseteq\NE(X)$ is automatic. By
Proposition~\ref{fibration-facets-prop}, the extremal rays of
$\mathcal C_\pi^\vee$ are generated by $F$ and by the classes $H_S$. The fiber
class is nef, so it remains to decide when every $H_S$ is nef.

Let $\Gamma\ne C_0$ be an irreducible horizontal curve, and put
\[
 m=\Gamma\cdot F>0,
 \qquad
 c=\Gamma\cdot C_0\ge0.
\]
Applying Lemma~\ref{local-correction-lem} in every reducible fiber gives
\[
 H_S\cdot\Gamma\ge c+m\bigl(\chi-\delta(S)\bigr).
\]
If $\delta(\pi)\le\chi$, this is nonnegative. The intersections with $C_0$
and with the components of the reducible fibers are nonnegative by
Proposition~\ref{supporting-class-prop}; every other irreducible vertical
curve has class $F$. Hence every $H_S$ is nef, and duality gives
$\mathcal C_\pi=\NE(X)$.

Conversely, assume $\mathcal C_\pi=\NE(X)$ and choose a selection $S$ with
$\delta(S)=\delta(\pi)$. Then $H_S$ is nef, so
Proposition~\ref{supporting-class-prop} gives
\[
 0\le H_S^2=\chi-\delta(\pi).
\]
This proves the equivalence.
\end{proof}

\subsection{A uniform semiampleness criterion}
Recall that $N(\pi)=\sum_i n_i$ is the total number of components of the
reducible fibers.

\begin{proposition}
\label{general-Artin-bound-prop}
Let $S$ be a selection, and let $Z$ be a nonzero effective integral divisor
supported on $\mathcal F_S$. Write
\[
 Z=aC_0+\sum_{i=1}^s V_i,
\]
where $V_i$ is supported on the components of $F_i$ which do not belong to
$S$. Then
\[
 Z^2+K_X\cdot Z
 \le -\chi a^2+(\chi-2)a
 +N(\pi)\left\lfloor\frac{a^2}{4}\right\rfloor.
\]
If $a=0$, then $Z^2\le-2$. Moreover, if $N(\pi)\le2\chi+3$, then every such
$Z$ satisfies
\[
 Z^2+K_X\cdot Z\le-2.
\]
\end{proposition}

\begin{proof}
Fix a reducible fiber $G=F_i$. Suppose first that the selected component is
not the identity component $E_G$. Removing it breaks the cycle into two paths
from $E_G$ to the missing component. Give the missing component coefficient
zero, orient both paths from $E_G$ toward it, and let $\Delta_e$ be the
coefficient at the initial vertex of an edge minus the coefficient at its
terminal vertex. If $b$ is the coefficient of $E_G$ in $V_i$, the changes on
each path add up to $b$, and
\[
 V_i^2=-\sum_e\Delta_e^2.
\]
Hence the contribution of this fiber to $Z^2+K_X\cdot Z$, apart from the
terms involving only $C_0$, is
\[
 2ab+V_i^2=\sum_e\Delta_e(a-\Delta_e).
\]
Every summand is at most $\lfloor a^2/4\rfloor$, and $G$ has $n_i$ edges. If
the selected component is $E_G$, then the section is disjoint from $V_i$ and
$V_i^2\le0$, so the same upper bound holds. Summing over all reducible fibers
and using $K_X\sim(\chi-2)F$ proves the first inequality.

If $a=0$, then $Z$ is supported on a disjoint union of proper chains in the
reducible fibers. Their intersection matrices are negative definite and even,
so $Z^2\le-2$.

Assume now that $a>0$ and $N(\pi)\le2\chi+3$. If $a=2u$, then $u\ge1$ and
\[
\begin{aligned}
 Z^2+K_X\cdot Z
 &\le -(4\chi-N(\pi))u^2+2(\chi-2)u\\
 &\le -(2\chi-3)u^2+2(\chi-2)u\le-u<0.
\end{aligned}
\]
The left-hand side is even, hence it is at most $-2$. If $a=2u+1$, then
\[
\begin{aligned}
 Z^2+K_X\cdot Z
 &\le -(4\chi-N(\pi))u(u+1)+2(\chi-2)u-2\\
 &\le -(2\chi-3)u(u+1)+2(\chi-2)u-2\le-2.
\end{aligned}
\]
This proves the final assertion.
\end{proof}

\begin{lemma}
\label{general-semiample-lem}
If $N(\pi)\le2\chi+3$, then
\[
 \Nef(X)=\mathrm{SAmp}(X).
\]
\end{lemma}

\begin{proof}
Since
\[
 \frac{\lfloor n_i^2/4\rfloor}{n_i}\le\frac{n_i}{4}
\]
for every reducible fiber, one has
\[
 \delta(\pi)\le\frac{N(\pi)}4<\chi.
\]
By Theorem~\ref{Jac-thm}(1), the Mori cone is $\mathcal C_\pi$, and every
nonvertical extremal nef ray is big. The vertical ray is generated by $F$ and
is semiample.

Let $H_S$ span a nonvertical extremal ray, and choose $m_S>0$ such that
$M_S:=m_SH_S$ is integral. Its null locus consists precisely of $C_0$ and the
components in $\mathcal T\setminus S$. Indeed, these curves are orthogonal to
$H_S$. Conversely, a curve orthogonal to a big and nef divisor has negative
self-intersection and therefore spans an extremal ray of the Mori cone. Since
$\NE(X)=\mathcal C_\pi$, it must be one of the displayed curves.

By Proposition~\ref{general-Artin-bound-prop}, every nonzero positive integral
cycle supported on this null locus has arithmetic genus at most zero. The
intersection form on each connected component is negative definite by the
Hodge index theorem. Artin's criterion \cite{Artin62,Artin66} therefore shows that every connected
component is the exceptional divisor of a rational singularity; equivalently,
$H^1(Z,\mathcal O_Z)=0$ for every positive cycle supported there. Choose a
block $B_S$ for $M_S$. Proposition~\ref{block-semi-prop} gives
$\mathbf B(M_S)=\varnothing$, so $H_S$ is semiample. The assertion now follows
from Proposition~\ref{NE-reduced-prop}.
\end{proof}

\begin{proof}[Proof of Theorem~\ref{Jac-thm}(2)]
If $N(\pi)\le2\chi+3$, then $\mathcal C_\pi=\NE(X)$ by part~(1), while
Lemma~\ref{general-semiample-lem} gives
$\Nef(X)=\mathrm{SAmp}(X)$. Since $q(X)=0$ by
Proposition~\ref{kod1-basic-prop}, Proposition~\ref{char-MDS-prop} shows that
$X$ is a Mori dream surface.
\end{proof}

\begin{proof}[Proof of Corollary~\ref{small-rho-cor}]
Let $s$ be the number of reducible fibers. By the Shioda--Tate formula,
\[
 \rho(X)-2
 =\sum_{i=1}^s(n_i-1)
 =N(\pi)-s,
\]
and hence
\[
 N(\pi)=\rho(X)-2+s.
\]
Since $n_i\ge2$ for every $i$, one has $s\le\rho(X)-2$. Therefore
\[
 N(\pi)\le2\rho(X)-4.
\]
If $\rho(X)\le\chi+3$, then
\[
 N(\pi)\le2\chi+2<2\chi+3,
\]
so Theorem~\ref{Jac-thm}(2) applies. The last assertion follows from
$\chi\ge3$.
\end{proof}

\begin{proof}[Proof of Corollary~\ref{one-fiber-cor}]
If the unique reducible fiber is of type $I_n$, then
\[
 N(\pi)=n,
 \qquad
 \delta(\pi)=\frac{\lfloor n^2/4\rfloor}{n}.
\]
The inequality $\delta(\pi)\le\chi$ is equivalent to $n\le4\chi$.
Indeed, this is immediate when $n$ is even. If $n$ is odd, then
\[
 \delta(\pi)=\frac n4-\frac{1}{4n},
\]
which is at most $\chi$ for $n\le4\chi$ and greater than $\chi$ for
$n\ge4\chi+1$. Both assertions now follow from
Theorem~\ref{Jac-thm}.
\end{proof}

\subsection{The one-fiber case}
We record the sharp form of the Artin bound for a single reducible fiber.
Assume that $\pi\colon X\to\mathbb P^1$ has finite Mordell--Weil group,
$\chi=\chi(\mathcal O_X)\ge3$, zero section $C_0$, and exactly one reducible
fiber
\[
 F_1=C_1+\cdots+C_n
\]
of type $I_n$, where $n\ge3$. The components are cyclically ordered and
numbered so that $C_0\cdot C_1=1$. Thus $C_0^2=-\chi$ and
$K_X\sim(\chi-2)F$. Set
\[
 \mathcal C_\pi=\sum_{i=0}^n\mathbb R_{\ge0}[C_i].
\]

The vertical facet of $\mathcal C_\pi$ is
$\sum_{i=1}^n\mathbb R_{\ge0}[C_i]$. For each $k=1,\ldots,n$, let
$\sigma_k$ be the nonvertical facet obtained by omitting $C_k$.

\begin{proposition}
\label{non-vertical-prop}
Every nonzero effective integral divisor supported on one of the facets
$\sigma_k$ has arithmetic genus at most zero if and only if
$n\le2\chi+3$. If this inequality holds, then
$D^2+D\cdot K_X\le-2$ for every nonzero effective integral divisor supported
on $\sigma_k$.
\end{proposition}

\begin{proof}
Suppose first that $n\le2\chi+3$. For the selection $S=\{C_k\}$ one has
$\mathcal F_S=\sigma_k$ and $N(\pi)=n$. Hence
Proposition~\ref{general-Artin-bound-prop} gives
\[
 D^2+K_X\cdot D\le-2
\]
for every nonzero effective integral divisor $D$ supported on $\sigma_k$.
Thus $p_a(D)\le0$.

Conversely, assume $n\ge2\chi+4$. Choose $C_k$ so that both paths from
$C_1$ to $C_k$ have at least $\chi+2$ edges. Write their initial segments
as
\[
 C_1=A_0,A_1,\ldots,A_{\chi+2},\qquad
 C_1=B_0,B_1,\ldots,B_{\chi+2}.
\]
Set
\[
 Z=2C_0+(\chi+2)C_1+
   \sum_{j=1}^{\chi+1}(\chi+2-j)(A_j+B_j).
\]
Along each path there are exactly $\chi+2$ nonzero drops, all equal to
one. Therefore
\[
 Z^2+K_X\cdot Z=(-2\chi-4)+(2\chi+4)=0,
\]
and hence $p_a(Z)=1$. This proves the converse.
\end{proof}

\subsection{A Picard-number-three variant}
The main theorem treats reducible fibers of type $I_n$. In Picard number
three, the same conclusion also holds for the other Kodaira fiber with two
components.

\begin{theorem}
Let $\pi\colon X\to\mathbb P^1$ be a Jacobian elliptic surface with
$\chi(\mathcal O_X)\ge3$ and $\rho(X)=3$. If $\pi$ has exactly one reducible
fiber, then $X$ is a Mori dream surface of Kodaira dimension one.
\end{theorem}

\begin{proof}
Set $\chi=\chi(\mathcal O_X)$. Shioda--Tate shows that the Mordell--Weil
group is finite and that the unique reducible fiber is of type $I_2$ or
$\mathrm{III}$. The first case follows from Theorem~\ref{Jac-thm}, so assume
that the fiber is of type $\mathrm{III}$. Write it as $F_1=C_1+C_2$, where
$C_0$ is the zero section and $C_0\cdot C_1=1$. Then
\[
 C_1^2=C_2^2=-2,\qquad C_1\cdot C_2=2.
\]
The inverse of the intersection matrix of $C_0,C_1,C_2$ is
\[
 \begin{pmatrix}
 0 & 1 & 1\\
 1 & \chi & \chi\\
 1 & \chi & \chi-\dfrac12
 \end{pmatrix},
\]
which is coefficientwise nonnegative. Hence
$\NE(X)=\sum_{i=0}^2\mathbb R_{\ge0}[C_i]$. The dual nef cone is generated by
$F=C_1+C_2$, by $H_1=C_0+\chi F$, and by
\[
 H_2=C_0+\chi C_1+\left(\chi-\frac12\right)C_2.
\]
The fiber class is semiample. The null locus of $H_1$ is
$C_0\sqcup C_2$. The divisor $C_0+C_2$ is a block for $H_1$, and both
components are rational; hence Proposition~\ref{block-semi-prop} shows that
$H_1$ is semiample. Choose an integral multiple $M_2$ of $H_2$. Its null locus is
$C_0\cup C_1$, and every nonzero positive cycle $Z=aC_0+bC_1$ satisfies
\[
 Z^2+K_X\cdot Z
 =-\chi a^2-2b^2+2ab+(\chi-2)a\le-2.
\]
Indeed, this is clear for $a=0$; for $a\ge1$, completing the square in $b$
gives an upper bound
$-(\chi-\tfrac12)a^2+(\chi-2)a\le-\tfrac32$, while the expression is even.
Artin's criterion shows that every positive cycle supported on this null
locus has vanishing first cohomology. Applying
Proposition~\ref{block-semi-prop} to a block for $M_2$ shows that $M_2$ is
semiample. Thus all extremal nef rays are semiample, and
Propositions~\ref{NE-reduced-prop} and~\ref{char-MDS-prop} finish the proof.
\end{proof}

\section{A big and nef divisor which is not semiample}
\label{I11-section}
The cycle used below is the genus-one cycle constructed in the proof of
Proposition~\ref{non-vertical-prop} for $(\chi,n)=(3,11)$.
Let $\mathcal S_{I_{11}}$ denote the locus of Jacobian elliptic surfaces with
$\chi(\mathcal O_X)=3$ and singular-fiber configuration $I_{11}+25I_1$, and
let $\widetilde{\mathcal S}_{I_{11}}\to\mathcal S_{I_{11}}$ be the finite
cover on which the components of the $I_{11}$-fiber are labeled.

\begin{lemma}
\label{I11-cycle-lem}
Let $X$ be a surface represented by a point of
$\widetilde{\mathcal S}_{I_{11}}$. Write the reducible fiber as
$F=C_1+\cdots+C_{11}$, with cyclically ordered components, and let $C_0$ be
the zero section, numbered so that $C_0\cdot C_1=1$. Write the two paths from
$C_1$ to $C_6$ as
\[
 C_1=A_0,A_1,\ldots,A_5=C_6,
 \qquad
 C_1=B_0,B_1,\ldots,B_6=C_6.
\]
Let $H_6$ be the rational divisor in the span of
$C_0,\ldots,C_{11}$ determined by $H_6\cdot C_i=\delta_{i6}$, and set
$L=11H_6$. Define
\[
 Z=2C_0+5C_1+\sum_{j=1}^{4}(5-j)(A_j+B_j),
 \qquad
 D=\sum_{j=1}^{5}j(A_j-B_j),
\]
and put $Z'=Z-C_0$. Then $L$ is integral, $L^2=33$, and
$L\cdot C_i=11\delta_{i6}$. Moreover, $Z$ is a minimally elliptic cycle,
\[
 H^0(Z',\mathcal O_{Z'})=\mathbb C,
 \qquad H^1(Z',\mathcal O_{Z'})=0,
 \qquad \mathcal O_{Z'}(D)\simeq\mathcal O_{Z'}.
\]
In addition, $\operatorname{Pic}^0(Z)\simeq\mathbb G_a$ and
\[
 \mathcal O_Z(2L)\simeq\mathcal O_Z(-D).
\]
\end{lemma}

\begin{proof}
For the selection $S=\{C_6\}$, Proposition~\ref{supporting-class-prop}
identifies $H_6$ with $H_S$. Using the inverse Cartan matrix displayed in the
proof of Lemma~\ref{local-correction-lem}, one obtains
\[
\begin{split}
L={}&11C_0+33C_1+27C_2+21C_3+15C_4+9C_5+3C_6\\
   &\quad+8C_7+13C_8+18C_9+23C_{10}+28C_{11}.
\end{split}
\]
Thus $L$ is integral, $L\cdot C_i=11\delta_{i6}$, and $L^2=33$. Since
$K_X\sim F$, comparison of coefficients gives
\begin{equation}
 2L=11(K_X+Z)-D.
 \label{eq:I11-relation}
\end{equation}
By the construction in Proposition~\ref{non-vertical-prop}, one has
$Z^2+K_X\cdot Z=0$. Since $K_X\sim F$ and $Z\cdot F=2$, it follows that
$K_X\cdot Z=2$, $Z^2=-2$, and $p_a(Z)=1$.

We claim that every proper nonzero positive integral subcycle $T<Z$ satisfies
$T^2+K_X\cdot T\le-2$. Write
\[
 T=aC_0+cC_1+\sum_{j=1}^4(x_jA_j+y_jB_j),
\]
where $0\le a\le2$, $0\le c\le5$, and
$0\le x_j,y_j\le5-j$. Put $x_0=y_0=c$, $x_5=y_5=0$, and set
$d_j=x_{j-1}-x_j$ and $e_j=y_{j-1}-y_j$. Then
$\sum d_j=\sum e_j=c$, and
\begin{equation}
 T^2+K_X\cdot T
 =-3a^2+a+2ac-\sum_{j=1}^5(d_j^2+e_j^2).
 \label{eq:I11-drops}
\end{equation}
If $a=0$ and $T\ne0$, the last sum is a positive even integer. If $a=1$,
the inequalities $\sum d_j^2\ge c$ and $\sum e_j^2\ge c$ give the same
bound. If $a=2$ and $c\le4$, the right-hand side is at most
$-10+2c\le-2$. Finally, let $a=2$ and $c=5$. Each sum of five squares is at
least $5$, with equality only when all five drops are equal to $1$. Since
$T<Z$, one sequence is not $(1,1,1,1,1)$; its sum of squares is odd and at
least $7$. Thus the last sum in \eqref{eq:I11-drops} is at least $12$, proving
the claim.

The class $L$ is orthogonal to every component of $Z$ and has positive square,
so the Hodge index theorem shows that the intersection matrix of the support
of $Z$ is negative definite. The claim therefore shows that $Z$ is minimally
elliptic in the sense of \cite{Laufer77}.

Let $p=C_0\cap C_1$. Removing one copy of $C_0$ gives
\begin{equation}
 0\longrightarrow\mathcal O_{C_0}(-Z')
 \longrightarrow\mathcal O_Z
 \longrightarrow\mathcal O_{Z'}
 \longrightarrow0.
 \label{eq:I11-exact}
\end{equation}
Since $Z'|_{C_0}=C_0|_{C_0}+5p$, one has
$\mathcal O_{C_0}(-Z')\simeq N^*_{C_0/X}(-5p)
\simeq\mathcal O_{\mathbb P^1}(-2)$. By \cite[Section~2]{Laufer77}, the
preceding inequalities imply $H^1(T,\mathcal O_T)=0$ for every proper positive
subcycle $0<T<Z$. In particular, $H^1(Z',\mathcal O_{Z'})=0$. Since
$\chi(\mathcal O_Z)=0$ and
$\chi(\mathcal O_{\mathbb P^1}(-2))=-1$, sequence
\eqref{eq:I11-exact} gives $\chi(\mathcal O_{Z'})=1$. Hence
$H^0(Z',\mathcal O_{Z'})=\mathbb C$, and the same sequence gives
$h^1(Z,\mathcal O_Z)=1$.

Let $R=Z'_{\mathrm{red}}$ and let $\mathcal J$ be the nilradical of
$\mathcal O_{Z'}$. Since both $Z'$ and $R$ are connected and
$H^1(Z',\mathcal O_{Z'})=H^1(R,\mathcal O_R)=0$, the sequence
$0\to\mathcal J\to\mathcal O_{Z'}\to\mathcal O_R\to0$ gives
$H^1(Z',\mathcal J)=0$. The logarithm identifies
$1+\mathcal J$ with $\mathcal J$, so
$\operatorname{Pic}(Z')\to\operatorname{Pic}(R)$ is injective. Since
$R$ is a tree of rational curves, a line bundle on $R$ is determined by its
multidegree. The divisor $D$ has degree zero on every component; hence
$\mathcal O_{Z'}(D)\simeq\mathcal O_{Z'}$.

The reduced support of $Z$ is a tree of rational curves, so
$\operatorname{Pic}^0(Z)$ has no torus or abelian part. Its tangent space is
$H^1(Z,\mathcal O_Z)$; hence it is a one-dimensional unipotent group and
$\operatorname{Pic}^0(Z)\simeq\mathbb G_a$. Furthermore,
$(K_X+Z)\cdot C=0$ for every component $C$ of $Z$, and adjunction for a
minimally elliptic cycle gives $\omega_Z\simeq\mathcal O_Z$
\cite{Laufer77}. Restricting \eqref{eq:I11-relation} to $Z$ therefore yields
$\mathcal O_Z(2L)\simeq\mathcal O_Z(-D)$.
\end{proof}

\begin{lemma}
\label{I11-obstruction-regularity-lem}
On $\widetilde{\mathcal S}_{I_{11}}$, the classes measuring whether
$\mathcal O_Z(D)$ is trivial form a regular section of a line bundle. In
particular, the locus
\[
 \mathcal U=\{s\in\widetilde{\mathcal S}_{I_{11}}:
 \mathcal O_{Z_s}(D_s)\not\simeq\mathcal O_{Z_s}\}
\]
is Zariski open.
\end{lemma}

\begin{proof}
It is enough to work on an \'etale chart $S$ carrying a family of
elliptic surfaces and the labeled relative curves. Let $\mathcal Z$,
$\mathcal Z'$, and $\mathcal D$ be the relative cycle and divisor obtained
from the formulas in Lemma~\ref{I11-cycle-lem}, and let
$q\colon\mathcal Z\to S$, $q'\colon\mathcal Z'\to S$, and
$g\colon\mathcal C_0\to S$ be the projections. The inclusion
$\mathcal Z'\subset\mathcal Z$ is square-zero, and its ideal is
\[
 \mathcal I\simeq
 N^*_{\mathcal C_0/\mathcal X}(-5\mathfrak p),
\]
where $\mathfrak p=\mathcal C_0\cap\mathcal C_1$. Each geometric fiber of
$\mathcal I$ is $\mathcal O_{\mathbb P^1}(-2)$. Cohomology and base change
therefore show that $g_*\mathcal I=0$ and that
$\mathcal E:=R^1g_*\mathcal I$ is an invertible sheaf whose formation
commutes with base change.

Set $\mathcal M=\mathcal O_{\mathcal Z}(\mathcal D)$. By
Lemma~\ref{I11-cycle-lem}, $\mathcal M|_{Z'_s}$ is trivial and
$h^0(Z'_s,\mathcal O_{Z'_s})=1$, $h^1(Z'_s,\mathcal O_{Z'_s})=0$ for every
geometric point $s$. Thus
$\mathcal N:=q'_*(\mathcal M|_{\mathcal Z'})$ is invertible, base change
holds, and evaluation identifies
$q'^*\mathcal N$ with $\mathcal M|_{\mathcal Z'}$. Consequently,
$\mathcal M_0:=\mathcal M\otimes q^*\mathcal N^{-1}$ has a canonical
trivialization on $\mathcal Z'$.

The sequence
$1\to1+\mathcal I\to\mathcal O_{\mathcal Z}^*
\to\mathcal O_{\mathcal Z'}^*\to1$, together with
$1+\mathcal I\simeq\mathcal I$, identifies the kernel of
$\operatorname{Pic}_{\mathcal Z/S}\to
\operatorname{Pic}_{\mathcal Z'/S}$ with the additive group associated with
$\mathcal E$. Hence $\mathcal M_0$ defines a regular section
$\eta(\mathcal D)\in H^0(S,\mathcal E)$, whose value at $s$ vanishes exactly
when $\mathcal O_{Z_s}(D_s)$ is trivial. These sections are compatible on
\'etale overlaps and therefore descend to the labeled cover. Their
nonvanishing locus is open.
\end{proof}

\begin{lemma}
\label{I11-pure-dimension-lem}
Every irreducible component of $\mathcal S_{I_{11}}$, and of its finite
labeled cover, has dimension $18$.
\end{lemma}

\begin{proof}
A Jacobian elliptic surface with fundamental line bundle
$\mathcal O_{\mathbb P^1}(3)$ has a short Weierstrass equation
$y^2=x^3+Ax+B$, with $A\in H^0(\mathbb P^1,\mathcal O(12))$ and
$B\in H^0(\mathbb P^1,\mathcal O(18))$. Thus the space of pairs $(A,B)$ has
dimension $32$, and adjoining a point $p\in\mathbb P^1$ gives dimension $33$.

Up to a nonzero constant, the discriminant is $\Delta=4A^3+27B^2$. The
condition $\operatorname{ord}_p(\Delta)\ge11$ is imposed by eleven jet
equations. Hence every component of the corresponding incidence locus has
dimension at least $22$. The conditions $A(p)\ne0$,
$\operatorname{ord}_p(\Delta)=11$, and simplicity of the remaining zeros of
$\Delta$ are open; they define the locus with configuration $I_{11}+25I_1$.

The changes of Weierstrass coordinates and of the coordinate on the base are
given by $\mathbb G_m\times\operatorname{PGL}_2$, of dimension $4$. On the
locus under consideration the stabilizers are finite, since the discriminant
has at least three distinct support points. It follows that every irreducible
component of $\mathcal S_{I_{11}}$ has dimension at least $22-4=18$. On the
other hand, \cite[Proposition~4.7]{Kloosterman07} gives
$\dim\mathcal S_{I_{11}}=30-12=18$. Thus the stratum is pure of dimension
$18$. A finite cover preserves dimensions of irreducible components.
\end{proof}

\begin{proof}[Proof of Theorem~\ref{I11-thm}]
We first exhibit a point of $\mathcal U$. Let $t$ be an affine coordinate on
$\mathbb P^1$, and consider
\begin{equation}
 y^2=x^3+\frac{1+t^4}{4}x^2+t^{11}q_7(t),
 \label{eq:I11-Weierstrass}
\end{equation}
where $q_7$ is general of degree $7$ and $q_7(0)\ne0$. The coefficients
$a_2=(1+t^4)/4$ and $a_6=t^{11}q_7(t)$ are sections of
$\mathcal O_{\mathbb P^1}(6)$ and $\mathcal O_{\mathbb P^1}(18)$,
respectively. Thus the fundamental line bundle is
$\mathcal O_{\mathbb P^1}(3)$ and $\chi(\mathcal O_X)=3$.

Up to nonzero constants, the Weierstrass invariants are
$c_4=(1+t^4)^2$ and
$\Delta=-t^{11}q_7(t)((1+t^4)^3+432t^{11}q_7(t))$. Since $c_4(0)\ne0$ and
$\operatorname{ord}_0(\Delta)=11$, the fiber over $0$ is of type $I_{11}$.
Simplicity of the remaining roots and coprimality with $c_4$ are open
conditions on $q_7$, expressed by the corresponding discriminant and
resultant, and they hold for general $q_7$. Thus the other twenty-five
singular fibers are of type $I_1$. For general $q_7$ the leading coefficient
of $\Delta$ is nonzero, so $\deg\Delta=36$ and there is no singular fiber at
infinity. Hence the configuration is $I_{11}+25I_1$, and
$K_X\sim F$, whence $\kappa(X)=1$.

We compute the obstruction of Lemma~\ref{I11-obstruction-regularity-lem} on
this surface. The ideal
$\mathcal I=\ker(\mathcal O_Z\to\mathcal O_{Z'})$ is
$N^*_{C_0/X}(-5p)\simeq\mathcal O_{\mathbb P^1}(-2)$. Locally at
$p=C_0\cap C_1$, if $\epsilon=0$ and $t=0$ are equations for $C_0$ and
$C_1$, then $Z$ and $Z'$ are defined by $\epsilon^2t^5=0$ and
$\epsilon t^5=0$, and $\mathcal I$ is generated by $\epsilon t^5$.
Lemma~\ref{I11-cycle-lem} gives
$\mathcal O_{Z'}(D)\simeq\mathcal O_{Z'}$, and its trivialization is unique
up to a nonzero constant.

Let
$q_A=C_1\cap A_1$ and $q_B=C_1\cap B_1$. Choose an analytic neighborhood
$U_0$ of $C_0\cup C_1$, and neighborhoods $U_A,U_B$ of the two tails, so that
$U_0\cap U_A\ne\varnothing$ near $q_A$,
$U_0\cap U_B\ne\varnothing$ near $q_B$, and
$U_A\cap U_B=\varnothing$.
\begin{figure}[H]
\centering
\begin{tikzpicture}[
x=.92cm,
y=.92cm,
curve/.style={
  draw=black,
  line width=.9pt,
  line cap=round
},
neigh/.style={
  dashed,
  line width=.8pt
},
lab/.style={
  font=\small,
  inner sep=1pt
},
slab/.style={
  font=\scriptsize,
  inner sep=1pt
},
mult/.style={
  circle,
  draw=red!75!black,
  text=red!75!black,
  fill=white,
  font=\scriptsize\bfseries,
  inner sep=1.4pt
}
]
\coordinate (p) at (0,0);
\coordinate (qA) at (2.70,1.45);
\coordinate (qB) at (2.70,-1.45);
\coordinate (u23) at (4.25,2.12);
\coordinate (u34) at (5.95,2.42);
\coordinate (u45) at (7.65,2.42);
\coordinate (u56) at (9.35,2.12);
\coordinate (l1110) at (4.25,-2.12);
\coordinate (l109) at (5.95,-2.42);
\coordinate (l98) at (7.65,-2.42);
\coordinate (l87) at (9.35,-2.12);
\coordinate (l76) at (10.85,-1.52);
\path[
  neigh,
  draw=green!50!black,
  fill=green!4
]
(-2.85,-1.85)
.. controls (-3.25,-.90) and (-3.25,.90) .. (-2.85,1.85)
.. controls (-1.60,2.35) and (-.10,1.90) .. (1.15,1.95)
.. controls (2.25,2.00) and (2.95,2.10) .. (3.22,1.68)
.. controls (3.40,1.32) and (3.08,.70) .. (2.83,.35)
.. controls (2.62,.08) and (2.62,-.08) .. (2.83,-.35)
.. controls (3.08,-.70) and (3.40,-1.32) .. (3.22,-1.68)
.. controls (2.95,-2.10) and (2.25,-2.00) .. (1.15,-1.95)
.. controls (-.10,-1.90) and (-1.60,-2.35) .. cycle;
\path[
  neigh,
  draw=blue!75!black,
  fill=blue!4
]
(2.25,1.00)
.. controls (2.85,3.25) and (4.90,3.65) .. (7.15,3.52)
.. controls (9.10,3.42) and (10.55,2.82) .. (11.18,1.72)
.. controls (11.42,1.28) and (11.28,.75) .. (10.78,.62)
.. controls (9.15,.28) and (7.40,.82) .. (5.45,.87)
.. controls (4.00,.90) and (2.85,.78) .. cycle;
\path[
  neigh,
  draw=magenta!70!blue,
  fill=magenta!4
]
(2.25,-1.00)
.. controls (2.85,-3.25) and (4.90,-3.65) .. (7.15,-3.52)
.. controls (9.10,-3.42) and (10.55,-2.82) .. (11.35,-1.72)
.. controls (11.62,-1.28) and (11.48,-.72) .. (10.92,-.58)
.. controls (9.25,-.25) and (7.42,-.82) .. (5.45,-.87)
.. controls (4.00,-.90) and (2.85,-.78) .. cycle;
\draw[curve]
(-2.45,-.12)
.. controls (-1.62,.66) and (-.68,.58) .. (p);
\draw[curve]
(qB)
.. controls (1.85,-.42) and (.92,-.08) .. (p)
.. controls (.92,.08) and (1.85,.42) .. (qA);
\draw[curve]
(qA)
.. controls (3.20,1.82) and (3.68,2.08) .. (u23);
\draw[curve]
(u23)
.. controls (4.78,2.40) and (5.38,2.48) .. (u34);
\draw[curve]
(u34)
.. controls (6.48,2.52) and (7.08,2.52) .. (u45);
\draw[curve]
(u45)
.. controls (8.18,2.48) and (8.80,2.35) .. (u56);
\draw[curve]
(u56)
.. controls (10.65,1.92) and (11.55,.85) .. (11.55,0)
.. controls (11.55,-.85) and (11.35,-1.25) .. (l76);
\draw[curve]
(l76)
.. controls (10.35,-1.82) and (9.85,-2.03) .. (l87);
\draw[curve]
(l87)
.. controls (8.83,-2.38) and (8.22,-2.48) .. (l98);
\draw[curve]
(l98)
.. controls (7.10,-2.52) and (6.48,-2.52) .. (l109);
\draw[curve]
(l109)
.. controls (5.38,-2.48) and (4.78,-2.40) .. (l1110);
\draw[curve]
(l1110)
.. controls (3.68,-2.08) and (3.20,-1.82) .. (qB);
\fill (p) circle (1.7pt);
\fill (qA) circle (1.7pt);
\fill (qB) circle (1.7pt);
\fill (u23) circle (1.7pt);
\fill (u34) circle (1.7pt);
\fill (u45) circle (1.7pt);
\fill (u56) circle (1.7pt);
\fill (l1110) circle (1.7pt);
\fill (l109) circle (1.7pt);
\fill (l98) circle (1.7pt);
\fill (l87) circle (1.7pt);
\fill (l76) circle (1.7pt);
\node[lab] at (-1.60,-0.2) {$C_0$};
\node[lab] at (1.10,0.6) {$C_1$};
\node[lab] at (3.45,1.42) {$C_2$};
\node[lab] at (5.10,1.98) {$C_3$};
\node[lab] at (6.80,2.03) {$C_4$};
\node[lab] at (8.50,1.88) {$C_5$};
\node[lab] at (12.05,.18) {$C_6$};
\node[lab] at (10.15,-1.52) {$C_7$};
\node[lab] at (8.52,-1.88) {$C_8$};
\node[lab] at (6.80,-1.98) {$C_9$};
\node[lab] at (5.10,-1.88) {$C_{10}$};
\node[lab] at (3.45,-1.42) {$C_{11}$};
\node[lab] at (3.45,2.45) {$A_1$};
\node[lab] at (5.10,2.95) {$A_2$};
\node[lab] at (6.80,3.05) {$A_3$};
\node[lab] at (8.50,2.85) {$A_4$};
\node[lab] at (3.45,-2.45) {$B_1$};
\node[lab] at (5.10,-2.9) {$B_2$};
\node[lab] at (6.80,-3) {$B_3$};
\node[lab] at (8.50,-2.95) {$B_4$};
\node[lab] at (10.15,-2.22) {$B_5$};
\node[slab] at (12.35,-.28) {$A_5=B_6$};
\node[slab] at (0,0.3) {$p$};
\node[slab] at (2.82,1.20) {$q_A$};
\node[slab] at (2.82,-1.20) {$q_B$};
\node[mult] at (-1.62,.40) {$2$};
\node[mult] at (1.68,.52) {$5$};
\node[mult] at (3.45,1.92) {$4$};
\node[mult] at (5.10,2.45) {$3$};
\node[mult] at (6.80,2.55) {$2$};
\node[mult] at (8.50,2.35) {$1$};
\node[mult] at (3.45,-1.92) {$4$};
\node[mult] at (5.10,-2.4) {$3$};
\node[mult] at (6.80,-2.5) {$2$};
\node[mult] at (8.50,-2.42) {$1$};
\node[text=green!50!black] at (-1.70,1.22) {$U_0$};
\node[text=blue!75!black] at (6.80,3.95) {$U_A$};
\node[text=magenta!70!blue] at (6.80,-3.95) {$U_B$};
\end{tikzpicture}
\caption{The cycle $Z$ and the open sets used to trivialize
$\mathcal O_{Z'}(D)$. The red numbers are the multiplicities in $Z$.}
\end{figure}
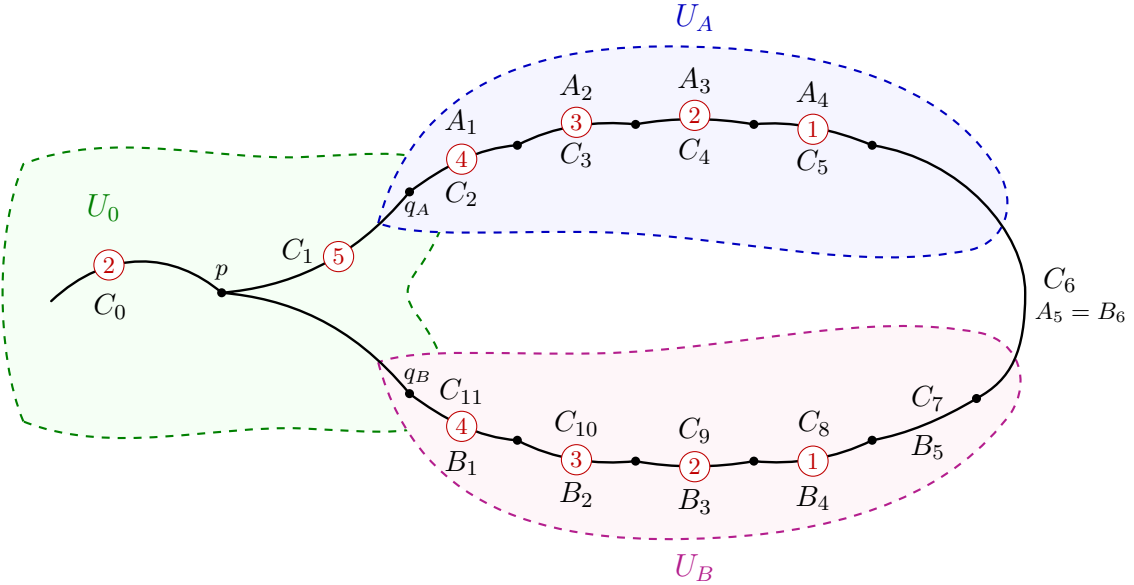
The exact sequence
\[
0\longrightarrow N^*_{C_0/X}(-5p)
\longrightarrow N^*_{C_0/X}
\longrightarrow N^*_{C_0/X}|_{5p}
\longrightarrow0
\]
gives, after choosing a local parameter $t$ at $p$ and a local trivialization
of the conormal bundle, an isomorphism
\begin{equation}
H^1\bigl(C_0,N^*_{C_0/X}(-5p)\bigr)
\simeq
\frac{\mathbb C[t]/(t^5)}{\langle1,t,t^2,t^3\rangle}.
\label{eq:I11-obstruction-space}
\end{equation}
This space is one-dimensional, generated by the class of $t^4$. 

Analytically at the node, the Weierstrass model has an
$A_{10}$-singularity $rs=t^{11}$. Numbering the branches suitably, the
pullback of $r$ has order $j$ along $A_j$, while the pullback of $s$ has order
$j$ along $B_j$, for $1\le j\le5$. Therefore
$D|_{U_A}=\operatorname{div}(r)$ and
$D|_{U_B}=-\operatorname{div}(s)$; the corresponding local generators of
$\mathcal O_X(D)$ are $r^{-1}$ and $s$.

It remains to determine the generator on $U_0\cap Z'$. Since $C_1$ has
multiplicity $5$ in $Z'$, we work modulo $t^5$, where the last term of
\eqref{eq:I11-Weierstrass} disappears. Put
$c(t)=\sqrt{1+t^4}$ and $w=2y/x$. The simultaneous normalization is then
given by
$x=(w^2-c(t)^2)/4$ and $y=w(w^2-c(t)^2)/8$. The two points over the node are
$w=c(t)$ and $w=-c(t)$. After exchanging the branches if necessary,
$w=c(t)$ corresponds to $q_A$ and $w=-c(t)$ to $q_B$. Hence
\[
 f=\frac{w+c(t)}{w-c(t)}.
\]
This function has a pole at $q_A$, a zero at $q_B$, and value $1$ at the zero section. Near
$q_A$, the functions $r$ and $w-c(t)$ differ by a unit, so $fr$ is
invertible; near $q_B$, the functions $s$ and $w+c(t)$ differ by a unit, so
$f/s$ is invertible. Thus $f$, $r^{-1}$, and $s$ glue to the unique normalized
trivialization of $\mathcal O_{Z'}(D)$.

Near $p$, the function $w$ has a simple pole along $C_0$. Setting
$\epsilon=1/w$ and working modulo $\epsilon^2$, we obtain
\[
 f=\frac{1+c(t)\epsilon}{1-c(t)\epsilon}
   \equiv1+2c(t)\epsilon.
\]
Relative to the generator $\epsilon t^5$ of $\mathcal I$, the resulting
principal part is $2c(t)/t^5$. Under
\eqref{eq:I11-obstruction-space}, the obstruction is therefore represented by
the jet of $2c(t)$ modulo $t^5$. Since
\[
 2c(t)=2\sqrt{1+t^4}\equiv2+t^4\pmod{t^5},
\]
the constant term lies in the image of
$H^0(\mathbb P^1,\mathcal O_{\mathbb P^1}(3))$, whereas $t^4$ does not. Hence
$\eta(D)=\pm[t^4]\ne0$. Thus $\mathcal O_Z(D)$ is nontrivial and, since
$\operatorname{Pic}^0(Z)\simeq\mathbb G_a$, it is non-torsion. In particular,
$\mathcal U$ is nonempty.

By Lemma~\ref{I11-obstruction-regularity-lem}, $\mathcal U$ is Zariski open.
Let $T$ be an irreducible component of the labeled stratum meeting
$\mathcal U$. Lemma~\ref{I11-pure-dimension-lem} gives $\dim T=18$. On the
other hand, \cite[Corollary~1.2]{Kloosterman07} gives
$\dim\operatorname{NL}_{13}=30-13=17$, where
$\operatorname{NL}_{13}$ denotes the locus of surfaces with Picard number at
least $13$. Its inverse image on the finite cover is a countable union of
closed subvarieties of dimension at most $17$, so it cannot contain the
nonempty open subset $\mathcal U\cap T$. Choose
\[
 X\in(\mathcal U\cap T)\setminus\operatorname{NL}_{13}.
\]
The trivial lattice of $X$ has rank $12$, while $\rho(X)<13$; hence
$\rho(X)=12$.

The Shioda--Tate formula now gives
$\operatorname{rank}\operatorname{MW}(\pi)=12-2-10=0$. Moreover,
\[
 \delta(\pi)=\frac{\lfloor 11^2/4\rfloor}{11}=\frac{30}{11}<3.
\]
Therefore Theorem~\ref{Jac-thm}(1) yields
\[
 \NE(X)=\sum_{i=0}^{11}\mathbb R_{\ge0}[C_i].
\]
Thus the Mori cone is simplicial. By Lemma~\ref{I11-cycle-lem}, the divisor
$L$ is integral, $L^2=33$, and it intersects the generators of the Mori cone
nonnegatively. Hence $L$ is big and nef.

Suppose that $L$ were semiample. Then $\mathcal O_X(mL)$ would be globally
generated for some $m>0$. Its restriction to the connected cycle $Z$ has
degree zero on every component. The associated morphism is constant on every
component of $Z_{\mathrm{red}}$, so its scheme-theoretic image is a connected
zero-dimensional scheme. Such a scheme is local Artinian, and every line
bundle on it is trivial. Therefore $\mathcal O_Z(mL)$ is trivial. This would
make $L|_Z$ torsion, contradicting
$\mathcal O_Z(2L)\simeq\mathcal O_Z(-D)$ and the non-torsion of
$\mathcal O_Z(D)$. Thus $L$ is not semiample. By
Proposition~\ref{char-MDS-prop}, $X$ is not a Mori dream surface.
\end{proof}

\section{Examples in every Picard number}
\label{Example-Sect}
\subsection{Picard number two}
We begin with a general observation.

\begin{theorem}
\label{Pic2-Jac-thm}
Let $\pi\colon X\to\mathbb P^1$ be a Jacobian elliptic surface with
$\chi(\mathcal O_X)\ge3$ and $\rho(X)=2$. Then $X$ is a Mori dream surface
of Kodaira dimension one.
\end{theorem}

\begin{proof}
Set $\chi=\chi(\mathcal O_X)$, let $F$ be a general fiber, and let $C_0$ be
the zero section. By Proposition~\ref{kod1-basic-prop}, one has
$q(X)=0$, $K_X\sim(\chi-2)F$, and $C_0^2=-\chi$. Shioda--Tate shows that
every fiber is irreducible and that $C_0,F$ form a basis of
$N^1(X)_{\mathbb Q}$.

If $\Gamma\ne C_0$ is an irreducible curve and
$\Gamma\equiv aC_0+bF$, then $a=\Gamma\cdot F\ge0$. If $a>0$, then
\[
 0\le\Gamma\cdot C_0=-a\chi+b,
\]
so $b\ge a\chi\ge0$; if $a=0$, then $\Gamma$ is vertical and its class is a
nonnegative multiple of $F$. Hence
\[
 \NE(X)=\mathbb R_{\ge0}[C_0]+\mathbb R_{\ge0}[F].
\]
The dual nef cone is generated by $F$ and $H=C_0+\chi F$. The first class is
semiample. The second is nef and big, with null locus $C_0$. Since $C_0$ is a
block for $H$ and $H^1(C_0,\mathcal O_{C_0})=0$,
Proposition~\ref{block-semi-prop} shows that $H$ is semiample. The result
follows from Propositions~\ref{NE-reduced-prop} and~\ref{char-MDS-prop}.
\end{proof}

\begin{theorem}
\label{Example-Pic2-thm}
For every integer $\chi\ge3$, there exists a Jacobian elliptic Mori dream
surface $\pi\colon X\to\mathbb P^1$ with $\kappa(X)=1$, $\rho(X)=2$, and
$\chi(\mathcal O_X)=\chi$ such that
\[
\NE(X)=\mathbb R_{\ge0}[F]+\mathbb R_{\ge0}[C_0],
\]
where $F$ is a general fiber and $C_0$ is the zero section of $\pi$.
\end{theorem}
\begin{proof}
Fix $\chi\ge3$ and put $p_g=\chi-1$. Cox \cite{Cox90} considered the
parameter space $\mathcal U$ of relatively minimal elliptic surfaces
$\pi\colon X\to\mathbb P^1$ with a section, geometric genus $p_g$, and
$q(X)=0$. As recalled in \cite[Section 2]{Ulmer17}, a very general member with
$\chi(\mathcal O_X)=p_g+1=\chi$ has no section other than the zero section
$O$, and
\[
\NS(X)=\mathbb Z[O]\oplus\mathbb Z[F],\qquad O^2=-\chi,\qquad O\cdot F=1,
\qquad F^2=0,\qquad K_X\sim(\chi-2)F.
\]
Note that $\kappa(X)=\kappa(X,F)=1$. Thus $X$ is a Mori dream surface by
Theorem~\ref{Pic2-Jac-thm}.
\end{proof}
\subsection{Higher Picard number}
\begin{theorem}
\label{Example-I_n-thm}
For every integer $\rho\ge3$, there is a Jacobian elliptic Mori dream surface
$\pi\colon X\to\mathbb P^1$ with $\kappa(X)=1$ and $\rho(X)=\rho$ such that
$\pi$ has exactly one reducible fiber of type $I_n$, where
$n=\rho-1\ge2$, and the Mordell--Weil group of $\pi$ is trivial.
\end{theorem}
\begin{proof}
Set $n=\rho-1\ge2$. Let \(S\) be a very general elliptic K3 surface with
\(\operatorname{NS}(S)\simeq U\). Its Jacobian elliptic fibration
\(h\colon S\to\mathbb P^1\) has only the zero section and, for a very general
choice, exactly twenty-four fibers of type \(I_1\). Fix one nodal value
\(b_0\in\mathbb P^1\) and a point \(p_0\in\mathbb P^1\).

Let \(\mathcal H_n\) be the space of degree-$n$ maps
\(f\colon\mathbb P^1\to\mathbb P^1\) satisfying
\(f^{-1}(b_0)=np_0\). After choosing coordinates with
\(p_0=b_0=\infty\), this is the irreducible space of polynomials
\[
 f(z)=a_nz^n+\cdots+a_0,\qquad a_n\ne0,
\]
so \(\dim\mathcal H_n=n+1\).

For every rational horizontal curve \(C\subset S\), let
\[
 q_C\colon\widetilde C\simeq\mathbb P^1\longrightarrow\mathbb P^1
\]
be the map induced by \(h\), and put \(m=\deg q_C\). Suppose that
\(f=q_C\circ a\) for a map \(a\) of degree \(d\). Then \(n=md\). If
\(m\ge2\), total ramification of \(f\) over \(b_0\) forces
\[
 q_C^{-1}(b_0)=mc,\qquad a^{-1}(c)=dp_0
\]
for some \(c\in\widetilde C\). Thus, after fixing coordinates, the possible
maps \(a\) form a space of polynomials of degree \(d\), of dimension
\(d+1\). The closure of the locus of maps in \(\mathcal H_n\) factoring
through this fixed \(q_C\) therefore has dimension at most
\[
 d+1=\frac nm+1\le\frac n2+1<n+1.
\]
The K3 surface \(S\) contains at most countably many rational curves: its Chow
variety has countably many components, and a positive-dimensional family of
rational curves would sweep out $S$, contradicting the fact that a complex K3
surface is not uniruled. For each fixed $C$, the dimension estimate above
places the factorization locus in a proper closed subset of $\mathcal H_n$.
Since an irreducible variety over the uncountable field $\mathbb C$ cannot be
covered by countably many proper closed subsets, their union does not cover
\(\mathcal H_n\). Choose \(f\) outside this union and also in the open set
where all other ramification is general and lies away from the remaining
twenty-three nodal values.

Let \(\pi\colon X\to\mathbb P^1\) be the Jacobian elliptic fibration on the
minimal desingularization of the base change. If \(s\) is a section of
\(\pi\), its image in \(S\) is a rational horizontal curve \(C\). Passing to
the normalization gives a factorization \(f=q_C\circ a\). By construction
this is impossible when \(m\ge2\). If \(m=1\), then \(C\) is a section of
\(h\), hence the zero section, and \(s\) is the induced zero section.
Therefore
\[
 \operatorname{MW}(\pi)=0.
\]

The totally ramified pullback of the chosen \(I_1\)-fiber is \(I_n\), while
the other nodal fibers give \(23n\) fibers of type \(I_1\). Consequently
\[
 c_2(X)=24n,\qquad \chi(\mathcal O_X)=2n,
 \qquad K_X\sim(2n-2)F.
\]
Thus \(\kappa(X)=1\) for \(n\ge2\), and
\(\rho(X)=n+1\) by Proposition~\ref{jac-rho-prop}. Since
$\chi(\mathcal O_X)=2n$, we have \(n\le2\chi(\mathcal O_X)+3\). Therefore
$X$ is a Mori dream surface by Theorem~\ref{Jac-thm}.
\end{proof}
\begin{proof}[Proof of Theorem~\ref{Mainthm}]
The proof follows from Theorems~\ref{Example-Pic2-thm}
and~\ref{Example-I_n-thm}.
\end{proof}

\end{document}